\documentclass{article}

\usepackage[utf8]{inputenc}
\usepackage{lmodern}

\usepackage{amsmath}
\usepackage{amsfonts}
\usepackage{amsthm}
\usepackage{mathtools}
\usepackage{paralist}
\usepackage{xcolor}
\usepackage{hyperref}
\usepackage[colorinlistoftodos]{todonotes}
\usepackage[affil-it]{authblk}

\usepackage[a4paper, margin=1.39in]{geometry}

\usepackage{biblatex}
\usepackage{bm}

\title{Structural convergence and algebraic roots\thanks{Supported by the European Research Council (ERC) under the European Union's Horizon 2020 research and innovation programme (ERC Synergy Grant DYNASNET, grant agreement No 810115).}
}
\author[1,2]{David Hartman}
\author[1]{Tom\'{a}\v{s} Hons}
\author[1]{Jaroslav Ne\v{s}et\v{r}il}
\affil[1]{\small Computer Science Institute, Faculty of Mathematics and Physics, Charles University, Prague, Czech Republic}
\affil[2]{\small Institute of Computer Science of the Czech Academy of Sciences, Prague, Czech Republic}
\date{}

\theoremstyle{plain}
\newtheorem{theorem}{Theorem}
\newtheorem{lemma}[theorem]{Lemma}
\newtheorem{proposition}[theorem]{Proposition}

\theoremstyle{definition}

\theoremstyle{remark}

\newcommand{\limply}{\rightarrow}  

\newcommand{\lneg}{\neg}

\newcommand{\bigland}{\bigwedge}

\DeclareMathOperator\Root{Root}

\newcommand{\N}{\mathbb{N}}
\newcommand{\Q}{\mathbb{Q}}
\newcommand{\R}{\mathbb{R}}

\DeclareMathOperator\FO{FO}
\DeclareMathOperator\QF{QF}

\newcommand{\stonepar}[2]{\langle #1, #2 \rangle}

\newcommand{\powerset}[1]{2^{#1}}
\newcommand{\defset}[2]{{#1}({#2})}

\newcommand{\eps}{\varepsilon}
\newcommand{\tpl}[1]{\bm{#1}}

\begin{document}

\maketitle

\begin{abstract}
Structural convergence is a framework for convergence of graphs by Ne\v{s}et\v{r}il and Ossona de Mendez that unifies the dense (left) graph convergence and Benjamini-Schramm convergence.
They posed a problem asking whether for a given sequence of graphs $(G_n)$ converging to a limit $L$ and a vertex $r$ of $L$ it is possible to find a sequence of vertices $(r_n)$ such that $L$ rooted at $r$ is the limit of the graphs $G_n$ rooted at $r_n$.
A counterexample was found by Christofides and Kr\'{a}l', but they showed that the statement holds for almost all vertices $r$ of $L$.
We offer another perspective to the original problem by considering the size of definable sets to which the root $r$ belongs.
We prove that if $r$ is an algebraic vertex (i.e. belongs to a finite definable set), the sequence of roots $(r_n)$ always exists.
\end{abstract}


\section{Introduction}\label{sec:introduction}

The field of graph convergence studies asymptotic properties of large graphs.
The goal is to define a well-behaved notion of a limit structure that describes the limit behavior of a convergent sequence of graphs.
Several different approaches are studied.
The two most prominent types of convergence are defined for sequences of dense \cite{counting_graph_homomorphism}\cite{limits_of_dense_graph_sequences}\cite{large_networks} and sparse graphs \cite{benjamini_schramm}\cite{elek}.
The recently introduced notion of structural convergence by Ne\v{s}et\v{r}il and Ossona de Mendez offers a generalizing framework for these cases using ideas from analysis, model theory and probability \cite{unified_approach}\cite{clustering}.

Structural convergence is a framework of convergence for general relational structures; however, we follow the usual approach that we restrict to the language of graphs and rooted graphs without loss of generality.
Our arguments remain valid in the general case (e.g. as in \cite{christofides_kral}).
The \emph{Stone pairing} of a first-order formula $\phi$ in the language of graphs and a finite graph $G$, denoted by $\stonepar{\phi}{G}$, is the probability that $\phi$ is satisfied by a tuple of vertices of $G$ selected uniformly at random (for a sentence $\phi$, we set $\stonepar{\phi}{G} = 1$ if $G \models \phi$, and $\stonepar{\phi}{G} = 0$ otherwise).
A sequence of finite graphs $(G_n)$ is said to be \emph{$\FO$-convergent} if the sequence $(\stonepar{\phi}{G_n})$ of probabilities converges for each formula $\phi$.
The limit structure $L$, called \emph{modeling}, is a graph on a standard Borel space with a probability measure $\nu$ satisfying that all the first-order definable sets are measurable.
The value $\stonepar{\phi}{L}$ is defined as the measure of the set $\defset{\phi}{L}$, the set of solutions of $\phi$ in $L$, using the appropriate power of the measure $\nu$.
A modeling $L$ is a limit of an $\FO$-convergent sequence $(G_n)$ if $\lim_n \stonepar{\phi}{G_n} = \stonepar{\phi}{L}$ for each formula $\phi$.
A modeling limit does not exist for each $\FO$-convergent sequence of finite graphs.
It is known to exist for all sequences of graphs from a class $\mathcal{C}$ if and only if $\mathcal{C}$ is a nowhere dense class \cite{existence_of_modeling_limits}.

The authors of this framework asked in \cite{unified_approach} the following question:
given a sequence $(G_n)$ converging to a modeling $L$ and a vertex of $r$ of $L$, is there a sequence of vertices $(r_n)$ such that the graphs $G_n$ rooted at $r_n$ converge to $L$ rooted at $r$?
Christofides and Kr\'{a}l' \cite{christofides_kral} provided an example that the answer is negative in general.
However, they also proved that it is always possible to find such a sequence $(r_n)$ for almost all choices of the vertex $r$.
That is, if the root of $L$ is chosen at random (according to the measure $\nu$), the vertices $(r_n)$ exist with probability $1$ \cite{christofides_kral}.

In this paper, we refine the original problem by considering the root $r$ to be an \emph{algebraic} vertex of $L$.
That is, $r$ belongs to a finite definable set of $L$ \cite{tent_ziegler}.
We prove that the sequence of roots $(r_n)$ always exists under such condition.
Our main result reads as follows:

\begin{theorem}\label{thm:rooting_algebraic_vertices}
    Let $(G_n)$ be an $\FO$-convergent sequence of graphs with a modeling limit $L$ and $r$ be an algebraic vertex of $L$.
    Then there is a sequence $(r_n)$, $r_n \in V(G_n)$, such that $(G_n, r_n)$ $\FO$-converges to $(L,r)$.
\end{theorem}

Note that Theorem~\ref{thm:rooting_algebraic_vertices} deals with full $\FO$-convergence and not just convergence with respect to sentences (called \emph{elementary convergence}), for which it is a trivial statement (see the case of $p=0$ in Lemma~\ref{lem:rooting_for_single_formula}).

In Section~\ref{sec:example} we formulate the example from \cite{christofides_kral} in our context to indicate that Theorem~\ref{thm:rooting_algebraic_vertices} is, in a way, best possible.
Moreover, we give a simple probabilistic construction of an $\FO$-convergent sequence of graphs that does not admit an $\FO$-convegent rooting when restricting the roots to a certain definable set.

This article is an extended version of the proceeding paper \cite{eurocomb_abstract}.

\section{Notation and tools}\label{sec:notation_and_tools}

We use $\N = \{1, 2, \dots \}, \N_0 = \N \cup \{0\}$ and $[n] = \{1, 2, \dots, n\}, [n]_0 = [n] \cup \{0\}$.
All graphs are finite except modelings, which are of size continuum.
The vertex set of a graph $G$ is denoted by $V(G)$.
The set of formulas in $p$ free variables in the language of graphs is denoted by $\FO_p$ and $\FO = \bigcup_{p \in \N_0} \FO_p$ is the set of all formulas.
Tuples of vertices, free variables, etc. are denoted by boldface letters, e.g. $\tpl{x} = (x_1, \dots, x_p)$.
Multiset is a set that allows multiplicities of its elements.
The power set of a set $X$ is denoted by $\powerset{X}$.

Let $G$ be an arbitrary graph and $r$ one of its vertices.
By $(G,r)$ we denote the graph $G$ rooted at $r$.
Formally, considering $G$ as a structure in the language of graphs, we add a new constant ``$\Root$'' to the vocabulary and interpret it as $r$.
We refer to the extended language as the language of rooted graphs.
The set of formulas in the extended language is denoted by $\FO^+$.
Note that $\FO_p \subseteq \FO^+_p$.
The observation that a rooted modeling is again a modeling was a motivation for the original problem of \cite{unified_approach}.

Let $L$ be a modeling.
A formula $\phi \in \FO_p$ is algebraic in $L$ if $\defset{\phi}{L}$ is finite, where $\defset{\phi}{L} = \{\tpl{v} \in V(L)^p : L \models \phi(\tpl{v})\}$ is the set of solutions of $\phi$ in $L$.
A vertex of $L$ is algebraic if it satisfies an algebraic formula \cite{tent_ziegler}.

We recall the Newton's identities (also known as Girard-Newton formulas) that connect sums of powers with symmetric polynomials.
One of the identities states that for given $a_1, \dots, a_n \in \R$, the coefficients of the polynomial $p(x) = \prod_{i=1}^n (x - a_i)$ can be obtained by basic arithmetic opetations from values $z_1, \dots, z_n$, where $z_k = \sum_{i=1}^n a_i^k$ \cite{enumerative_combinatorics}.

It is a folklore that the roots of a polynomial continuously depend on the coefficients of the polynomial:
for a polynomial $p(x) = \prod_{i=1}^n (x - a_i) = \sum_{j=0}^n c_j x^j$ and $\eps > 0$ there is $\delta > 0$ such that each polynomial $q(x) = \sum_{j=0}^n d_j x^j$ with $|c_j - d_j| < \delta$ can decomposed as $\prod_{i=1}^n (x - b_i)$ satisfying $|a_i - b_i| < \eps$.
All the coefficients and roots are complex \cite{complex_analytic_varieties}.

These classical tools will be used in the proof of Theorem~\ref{thm:rooting_algebraic_vertices} and particularly in the key Lemma~\ref{lem:probability_on_boolean_lattices} in Section~\ref{sec:finite_boolean_lattice}.
From Newton's identities it follows that the sums $z_1, \dots, z_n$ determine the values $a_1, \dots, a_n$ up to a permutation, which is the fact we utilize.
The continuous dependence of roots on coefficients is used for polynomials created via Newton's identities with real values.
Then the statement reads as follows: for each $a_1, \dots, a_n \in \R$ and $\eps > 0$ there is $\delta > 0$ such that for each $b_1, \dots, b_n \in \R$ with $\left|\sum_{i=1}^n a_i^k - \sum_{i=1}^n b_i^k \right| < \delta$ holds that there is a permutation $\pi$ satisfying $|a_i - b_{\pi(i)}| < \eps$.

\section{Rooting in algebraic sets}\label{sec:rooting_in_algebraic_sets}

We prove Theorem~\ref{thm:rooting_algebraic_vertices} in the following equivalent form.

\begin{theorem}\label{thm:rooting_solution_of_algebraic_formulas}
    Let $(G_n)$ be an $\FO$-convergent sequence of graphs with a modeling limit $L$ and $\xi(x)$ be an algebraic formula in $L$.
    Then there is a sequence $(r_n)$, $r_n \in V(G_n)$, and a vertex $r \in \defset{\xi}{L}$ such that $(G_n, r_n)$ $\FO$-converges to $(L,r)$.
\end{theorem}

Obviously, Theorem~\ref{thm:rooting_solution_of_algebraic_formulas} is implied by Theorem~\ref{thm:rooting_algebraic_vertices}.
The converse follows from fact that $\xi$ has only finitely many solutions in $L$ and we can root them iteratively one by one until we reach $r$.

Fix $(G_n)$, $L$, and $\xi$ for the rest of the paper.
Without loss of generality, assume that $\defset{\xi}{L}$ is an inclusion-minimal definable set in $L$ and $|\defset{\xi}{G_n}| = |\defset{\xi}{L}|$ for each $n$.
We prove Theorem~\ref{thm:rooting_solution_of_algebraic_formulas} in three steps.
First, we consider a single formula $\phi$ in the language of rooted graphs and show that we can find the roots $(r_n)$ and $r$ such that $\lim \stonepar{\phi}{(G_n, r_n)} = \stonepar{\phi}{(L, r)}$.
Then we consider an arbitrary finite collection of formulas $\phi_1, \dots, \phi_k$ and construct a single formula $\psi$ with the property that convergence of $\stonepar{\psi}{(G_n, r_n)}$ to $\stonepar{\psi}{(L, r)}$ implies convergence of each $\stonepar{\phi_i}{(G_n, r_n)}$ to $\stonepar{\phi_i}{(L, r)}$.
Finally, a routine use of compactness extends the previous to all formulas, which proves the theorem.

\subsection{Single formula}\label{ssec:algebraic_vertex_single_formula}

For a formula $\phi(\tpl{x}) \in \FO^+_p$, let $\phi^-(\tpl{x}, y) \in \FO_{p+1}$ be the formula created from $\phi$ by replacing each occurrence of the term ``$\Root$'' by ``$y$'' (we assume that $y$ does not appear in $\phi$).

\begin{lemma}\label{lem:rooting_for_single_formula}
    For a given $\phi \in \FO^+_p$ there is a sequence $(r_n)$, $r_n \in \defset{\xi}{G_n}$, and a vertex $r \in \defset{\xi}{L}$ such that $\lim \stonepar{\phi}{(G_n, r_n)} = \stonepar{\phi}{(L,r)}$.
\end{lemma}
\begin{proof}
    If $p = 0$, then either the sentence $(\forall y)(\xi(y) \limply \phi^-(y))$ or $(\forall y)(\xi(y) \limply \lneg\phi^-(y))$ is satisfied in $L$ (using the assumption that $\defset{\xi}{L}$ is an inclusion-minimal definable set); hence, it holds in each $G_n$ from a certain index on.
    Therefore, an arbitrary choice of $r_n \in \defset{\xi}{G_n}$ and $r \in \defset{\xi}{L}$ meets the conclusion.

    Let $\nu$ be the measure associated to the modeling $L$.
    Define $f_L: V(L)^p \to \powerset{\defset{\xi}{L}}$ to be the function that sends $\tpl{v}$ to the set $\{u \in \defset{\xi}{L} : L \models \phi^-(\tpl{v}, u)\}$.
    Consider the pushforward measure $\mu_L$ on $\powerset{\defset{\xi}{L}}$ of the $p$-th power of $\nu$ by $f_L$ (note that for each $u \in \xi(L)$ is the set $f_L^{-1}({u})$ measurable).
    Viewing $\powerset{\defset{\xi}{L}}$ as a lattice, we are mostly interested in the measure of the filter generated by atoms of $\powerset{\defset{\xi}{L}}$.
    Let $X^\uparrow$ denote the filter generated by $X \in \powerset{\defset{\xi}{L}}$.
    Observe that for $u \in \defset{\xi}{L}$ we have $\mu_L(\{u\}^\uparrow) = \stonepar{\phi}{(L, u)}$.
    Suppose that $|\defset{\xi}{L}| = t$ and define an ordering $R_L = (u_1, u_2, \dots, u_t)$ such that $\mu_L(R_L) = (\mu_L(\{u_i\}^\uparrow))_{i \in [t]}$ satisfies $\mu_L(\{u_1\}^\uparrow) \geq \mu_L(\{u_2\}^\uparrow) \geq \dots \geq \mu_L(\{u_t\}^\uparrow)$.
    Define similarly for each $n$ the function $f_n: V(G_n)^p \to \powerset{\defset{\xi}{G_n}}$, measure $\mu_n$ (as the pushforward of the uniform measure) and the vector $R_n$.
    
    We prove that the sequence $\big(\mu_n(R_n)\big) \subset ([0,1]^t, \|\cdot\|_\infty)$ converges to $\mu_L(R_L)$.
    Then an arbitrary choice of an index $i \in [t]$ yields the sequence $(r_n)$ and vertex $r$ as the $i$-th elements of the vectors $R_n$, resp. $R_L$.
    
    This follows from Lemma~\ref{lem:probability_on_boolean_lattices} below applied for the set $M = \powerset{\defset{\xi}{L}}$ and the probability distribution $\mu_L$.
    Observe that the probabilities $\Pr[F_\ell^k]$ are given by $\stonepar{\psi_{k,\ell}}{L}$, where $\psi_{k,\ell}(\tpl{x}_1, \dots, \tpl{x}_k) \in \FO_{k \cdot p}$ is 
    \[
        (\exists y_1, \dots, y_\ell)
        \left(
            \bigland_{i = 1}^l \xi(y_i)
            \land
            \bigland_{1 \leq i < j \leq \ell} y_i \not= y_j
            \land
            \bigland_{i = 1}^k \bigland_{j = 1}^\ell \phi^-(\tpl{x}_i, y_j)
        \right)
        .
    \]
    Due to the continuity part of Lemma~\ref{lem:probability_on_boolean_lattices}, as $\stonepar{\psi_{k,\ell}}{G_n} \to \stonepar{\psi_{k,\ell}}{L}$, we reach the conclusion.
\end{proof}

\subsection{Finite collection of formulas}\label{ssec:algebraic_vertex_finite_collection}

In this part, we use Lemma~\ref{lem:rooting_for_single_formula} to prove an analogous statement for a finite collection of formulas.

\begin{lemma}\label{lem:rooting_for_finite_collection_of_formulas}
    For given formulas $\phi_1, \dots, \phi_k$ there is a sequence $(r_n)$, $r_n \in \defset{\xi}{G_n}$, and a vertex $r \in \defset{\xi}{L}$ such that $\lim \stonepar{\phi_i}{(G_n, r_n)} = \stonepar{\phi_i}{(L,r)}$ for each $\phi_i$.
\end{lemma}
\begin{proof}
    Since for sentences any choice of $(r_n)$ and $r$ works, we assume that neither of $\phi_1, \dots, \phi_k$ is a sentence.
    
    Consider an inclusion-maximal set $I \subseteq [k]$ for which there is $v \in \defset{\xi}{L}$ such that every $i \in I$ satisfies $\stonepar{\phi_i}{(L,v)} > 0$.
    Denote $|I|$ by $k'$.
    If $I = \emptyset$, we can choose $(r_n)$ and $r$ arbitrarily; hence, assume otherwise.
    For $i \in I$ set $A_i = \{\stonepar{\phi_i}{(L,u)} : u \in \defset{\xi}{L} \} \cap (0, 1]$.
    Take a vector $\tpl{e} \in \N^{k'}$ of exponents with the property that for each distinct $\tpl{a}, \tpl{b} \in \bigtimes_{i \in I} A_i$ we have $\prod_{i \in I} a_i^{e_i} \not= \prod_{i \in I} b_i^{e_i}$.
    Such a vector exists as each $A_i$ is finite and contains only positive values.
    The set of \emph{bad} choices of rational exponents that make the values for particular $\tpl{a}, \tpl{b}$ coincide form a $(k'-1)$-dimensional hyperplane in $\Q^{k'}$.
    We can surely avoid finitely many of such hyperplanes (one for each choice of $\tpl{a}$ and $\tpl{b}$) to find a \emph{good} vector of positive rational exponents and scale them to integers.
    
    Use Lemma~\ref{lem:rooting_for_single_formula} for the formula $\psi$ of the form
    \[
        \bigland_{i \in I} \bigland_{j=1}^{e_i} \phi_i(\tpl{x}_{i,j})
        ,
    \]
    where all the tuples $\tpl{x}_{i,j}$ are pairwise disjoint, to obtain roots $(r_n)$ and $r$.
    In particular, we can take the vertex $r$ such that $\stonepar{\psi}{(L,r)} > 0$ (due to our choice of $I$).

    We have $\lim \stonepar{\phi_i}{(G_n, r_n)} = \stonepar{\phi_i}{(L, r)} > 0$ for each $i \in I$ as
    \[
        \stonepar{\psi}{(L,r)} = \prod_{i \in I}  \stonepar{\phi_i}{(L,r)}^{e_i}
        ,
    \]
    using our selection of exponents $\tpl{e}$.
    
    Also, it holds that $\lim \stonepar{\phi_j}{(G_n, r_n)} = \stonepar{\phi_j}{(L,r)} = 0$ for each $j \not\in I$: 
    for the formula $\chi = \bigland_{i \in I \cup \{j\}} \phi_i(\tpl{x}_i)$, we have $\lim \stonepar{\chi}{(G_n,r_n)} = \stonepar{\chi}{(L,r)} = 0$ due to the maximality of $I$ (this is for \emph{any} choice of $(r_n)$ and $r$).
    We have
    \[
        \stonepar{\chi}{(G_n,r_n)} = \prod_{i \in I \cup \{j\}}  \stonepar{\phi_i}{(G_n,r_n)}
    \]
    and as for some $\eps > 0$ there is $n_0$ such that $\stonepar{\phi_i}{(G_n,r_n)} > \eps$ for each $i \in I$ and $n \geq n_0$, the factor $\stonepar{\phi_j}{(G_n,r_n)}$ must tend to $0$.   
\end{proof}

We remark that the rationalization of the fact that the sequence $\big(\stonepar{\phi_j}{(G_n, r_n)}\big)$ for $j \not\in I$ even converge is the reason why we are proving Theorem~\ref{thm:rooting_solution_of_algebraic_formulas} instead of Theorem~\ref{thm:rooting_algebraic_vertices}.
That is, we are using the fact that we can choose the set $I$ (and the root $r$ for the formula $\psi$) such that any rooting $(r_n)$ makes the sequence $\stonepar{\chi}{(G_n, r_n)}$ converge to $0$.

\subsection{All formulas}\label{ssec:algebraic_vertex_all_formulas}

Fix an arbitrary ordering $\phi_1, \phi_2, \dots$ of $\FO^+$.
We call a collection of sequences of roots $(r_n^i)$, $r_n^i \in \defset{\xi}{G_n}$, and $r^i \in \defset{\xi}{L}$ \emph{extending} if for all indices $j \leq i \in \N$ satisfies that $\lim \stonepar{\phi_j}{(G_n, r_n^i)} = \stonepar{\phi_j}{(L,r^i)}$ and, moreover, for any $i' \geq i$ it holds $\stonepar{\phi_j}{(L,r^i)} = \stonepar{\phi_j}{(L,r^{i'})}$.
We prove existence of such a collection and then extract the desired sequence $(r_n)$ and vertex $r$ by diagonalization.

\begin{lemma}\label{lem:collection_of_extending_rootings}
    There exists an extending collection of sequences of roots $(r_n^i)$, $r_n^i \in \defset{\xi}{G_n}$, and $r^i \in \defset{\xi}{L}$.
\end{lemma}
\begin{proof}
    Let $S_j$ be the set of possible limit values of $\stonepar{\phi_j}{(G_n, r_n)}$, i.e. the values of $\mu_L(R_L)$ from Lemma~\ref{lem:rooting_for_single_formula}.
    Let $T = (V,E)$ be an infinite rooted tree defined as follows:
    $V = \bigcup V_\ell$, where $V_\ell$ is the set of vertices on the level $\ell$ defined as the Cartesian product of the sets $S_j$ for $j \leq \ell$, i.e. a vertex on the $\ell$-th level is a vector with possible limit probabilities for $\phi_1, \dots, \phi_\ell$.
    The only element in $V_0$, the empty set, is the root of the tree.
    We put an edge between $\tpl{a} \in V_\ell$ and $\tpl{b} \in V_{\ell+1}$ if $a_j = b_j$ for all $j \leq \ell$ and there are vertices $x_n \in \defset{\xi}{G_n}, x \in \defset{\xi}{L}$ such that
    \[
        \lim \stonepar{\phi_j}{(G_n, x_n)} = \stonepar{\phi_j}{(L, x)} = b_j
    \]
    for each $j \leq \ell+1$.
    Observe that if $\tpl{a}\tpl{b} \in E$, then $\tpl{b}$ is connected to the root (all edges of the path are witnessed by the vertices $(x_n)$ and $x$).
    
    By Lemma~\ref{lem:rooting_for_finite_collection_of_formulas}, there is at least one vertex in each set $V_\ell$ connected to the root.
    Thus, by K\"{o}nig's lemma, the tree contains an infinite path $\emptyset = \tpl{a}_0, \tpl{a}_1, \tpl{a}_2, \dots$ (all degrees are bounded as each $S_j$ is finite).
    The sequence $(r_n^i)$ and vertex $r^i$ are defined as the vertices $(x_n)$ and $x$ witnessing the edge from $\tpl{a}_{i-1}$ to $\tpl{a}_i$.
\end{proof}

Now we are ready to give the proof of Theorem~\ref{thm:rooting_solution_of_algebraic_formulas}.

\begin{proof}[Proof of Theorem~\ref{thm:rooting_solution_of_algebraic_formulas}]
    Let $(r_n^i)$ and $r^i$ for $i \in \N$ be the vertices from collection of extending sequences from Lemma~\ref{lem:collection_of_extending_rootings} above.
    Let $\ell_i = \lim \stonepar{\phi_i}{(L,r^i)}$ and let $N_i$ to be an index satisfying
    \begin{enumerate}[(i)]
        \item $\stonepar{\phi_j}{(G_n, r_n^i)} \in (\ell_j - 2^{-i}, \ell_j + 2^{-i})$ for each $n \geq N_i$ and $j \leq i$,
        \item $N_i > N_j$ for each $j < i$.
    \end{enumerate}
    Set $r_n = r_n^i$, where $i$ is the minimal positive integer satisfying that $n < N_{i+1}$.
    For the vertex $r$, we can set an arbitrary vertex from $\defset{\xi}{L}$ that appears infinitely many times as $r^i$.
    
    It remains to verify that for an arbitrary formula $\phi_j$ we have $\lim \stonepar{\phi_j}{(G_n, r_n)} = \ell_j = \stonepar{\phi_j}{(L, r)}$.
    Obviously, the second equality holds as $r = r^i$ for some $i \geq j$.
    For the first equality, fix $\eps > 0$ and set $k$ to be a positive integer satisfying $2^{-k} < \eps$ and $k \geq \max\{j, 2\}$.
    Then for $n \geq N_k$ we have $\stonepar{\phi_j}{(G_n, r_n)} = \stonepar{\phi_j}{(G_n, r_n^i)} \in (\ell_j - 2^{-i}, \ell_j + 2^{-i})$ for some $i \geq k$.
    Thus $\stonepar{\phi_j}{(G_n, r_n)} \in (\ell_j - \eps, \ell_j + \eps)$, which concludes the proof.
\end{proof}

\section{A lemma about finite boolean lattices}\label{sec:finite_boolean_lattice}

Let $M$ be a finite set of size $m$.
We are going to work with a probability space that allows us to consider random subsets $S_i$ of $M$ with $\Pr[S_i = X] = \mu(\{X\})$ for each $X \subseteq M$, where $\mu$ is a probability distribution on $\powerset{M}$.
Let $S_1, S_2, \dots, S_k$ be independently chosen random subsets of $M$ with distribution $\mu$, denote by $E_X^k$ the event $X \subseteq \bigcap_{i=1}^k S_i$ and by $F_\ell^k$ the event $|\bigcap_{i=1}^k S_i| \geq \ell$.
Let $X^\uparrow$ stand for the filter $\{Y \in \powerset{M} : X \subseteq Y\}$ and $\mathcal{M}_\ell$ denote the row of $\ell$-element sets, i.e. the set $\{X \subseteq M : |X|= \ell\}$.
For $\ell \in [m]_0$, we define the multiset $A_\ell = \{ \mu(X^\uparrow) : X \in \mathcal{M}_\ell \}$.

\begin{lemma}\label{lem:probability_on_boolean_lattices}
    The values of $\Pr[F_\ell^k]$ for $\ell \in [m]_0$ and $k \in \left[ \binom{m}{\ell} \right]$ continuously determine the multisets $A_\ell$ for all $\ell \in [m]_0$.
    
    That is, for each $\eps > 0$ there is $\delta > 0$ such that changing each $\Pr[F_\ell^k]$ by at most $\delta$ induces a change of values in $A_\ell$ by at most $\eps$.
\end{lemma}
\begin{proof}
    We proceed by downward induction on $\ell$.
    For $\ell = m$, we have $\mu(M^\uparrow) = \Pr[F_m^1]$.
    
    Now fix $\ell < m$ and suppose that all $A_{\ell'}$ for $\ell' > \ell$ are known.
    We have that
    \[
        \Pr[F_\ell^k]
        =
        \Pr
        \left[ 
            \bigcup_{X \in \mathcal{M}_\ell} E_X^k
        \right]
        .
    \]

    We apply the inclusion-exclusion principle on the union of events.
    Observe that for $\mathcal{I} \subseteq \powerset{M}$ we have $\bigcap_{X \in \mathcal{I}} E_X^k = E_{\bigcup \mathcal{I}}^k$, where $\bigcup \mathcal{I}$ stands for $\bigcup_{X \in \mathcal{I}} X$.
    \[
        \Pr
        \left[ 
            \bigcup_{X \in \mathcal{M}_\ell} E_X^k
        \right]
        =
        \sum_{j=1}^{\binom{m}{\ell}} (-1)^{j-1} \sum_{\mathcal{I} \subseteq \mathcal{M}_\ell: |\mathcal{I}| = j}
        \Pr
        \left[ 
            E_{\bigcup \mathcal{I}}^k
        \right]
    \]
    Now we gather the terms with the same set $Y = \bigcup \mathcal{I}$ together.
    Let $C(j,\ell,r)$ be number of covers of the set $[r]$ by $j$ distinct subsets of size $\ell$ and define
    \[
        D(\ell,r) = \sum_{j=1}^{\binom{r}{\ell}} (-1)^{j-1} C(j,\ell,r)
        .
    \]
    Then we have
    \begin{align*}
        \sum_{j=1}^{\binom{m}{\ell}} (-1)^{j-1} \sum_{\mathcal{I} \subseteq \mathcal{M}_\ell: |\mathcal{I}| = j}
        \Pr
        \left[ 
            E_{\bigcup \mathcal{I}}^k
        \right]
        =
        \sum_{r=\ell}^m (-1)^{r-\ell} \sum_{Y \in \mathcal{M}_r} D(\ell,r)
        \Pr
        \left[ 
            E_Y^k
        \right]
        .
    \end{align*}
    Moving the known terms to the left-hand side, using that $D(\ell, \ell) = 1$ and $\Pr[E_Y^k] = \mu(Y^\uparrow)^k$, we obtain
    \begin{equation}\label{eq:last_step_in_lattice_lemma}
        \Pr[F_\ell^k] - \sum_{r=\ell+1}^m (-1)^{r-\ell} \sum_{Y \in \mathcal{M}_r} D(\ell,r) \mu(Y^\uparrow)^k
        =
        \sum_{Y \in \mathcal{M}_\ell} \mu(Y^\uparrow)^k
        ,
    \end{equation}
    from which we determine the multiset $A_\ell$ using Newton's identities.
    The identities are applicable as Equation~\eqref{eq:last_step_in_lattice_lemma} holds for all $k \in \left[ \binom{m}{\ell} \right]$ and all the values on the left-hand side are known by induction hypothesis (strictly speaking, we know the values $\mu(Y^\uparrow)^k$ only up to a permutation of $\mathcal{M}_r$, but the formula is symmetric).
    
    Finally, we argue that the values in $A_\ell$ continuously depend on $\Pr[F_\ell^k]$.
    This is proved by induction.
    The case $\ell = m$ is obvious.
    For $\ell < m$, the left-hand side of \eqref{eq:last_step_in_lattice_lemma} is a continuous function of values $\Pr[F_\ell^k]$ and $\mu(Y^\uparrow)^k$ for $Y \in \mathcal{M}_r$, $r > \ell$.
    The later terms continuously depend on $\Pr[F_\ell^k]$ by induction hypothesis.
    Obtaining the values in $A_\ell$ via Newton's identities is also a continuous process as the complex roots of a complex polynomial continuously depend on its coefficients.
\end{proof}

\section{Examples}\label{sec:example}

The original example with bipartite graphs of Christofides and Kr\'{a}l' \cite{christofides_kral} implies that if the root $r$ of a modeling does not belong to a finite definable set, the desired sequence of roots $(r_n)$ needs not to exist.
Moreover, observe that the root $r$ lies in a countable definable set (vertices from the smaller part $B$ can be distinguished by the property that they have no twin, i.e. another vertex with the same neighborhood).
Therefore, the finiteness of the definable set is the weakest sufficient condition for the sequence $(r_n)$ to exists regarding the cardinality of definable sets containing $r$. 

Here we give a simple probabilistic construction for the following statement.

\begin{proposition}\label{prop:example}
    There exists an $\FO$-convergent sequence of graphs $(G_n)$ and a formula $\xi(x)$ (satisfying $G_n \models (\exists x)\xi(x)$ for all $n$) with the property that there are no roots $(r_n)$, $r_n \in \defset{\xi}{G_n}$, such that the sequence $(G_n, r_n)$ is $\FO$-convergent.
\end{proposition}

Note that necessarily $\lim \stonepar{\xi}{G_n} = 0$, otherwise the roots exists by the result Christofides and Kr\'{a}l' \cite{christofides_kral}.
Moreover, the sequence $|\defset{\xi}{G_n}|$ has to be unbounded, otherwise the roots exists by Theorem~\ref{thm:rooting_solution_of_algebraic_formulas} (see the second paragraph the last section).

It is an interesting question whether the sequence in Proposition~\ref{prop:example} can be chosen from a nowhere dense class \cite{sparsity}.

\subsection{Proof of Proposition~\ref{prop:example}}

First we describe the example with bipartite graphs with distinguished parts (i.e. marked by distinct unary symbols), then we show how to remove the marks to obtain a sequence of simple graphs.

Let $L = \{E,A,B\}$ be a language with one binary relation $E$ and two unary relations $A, B$.
Let $T$ be the theory of bipartite graph with vertices in the first part marked by the symbol $A$ and vertices in the other part marked by the symbol $B$.
We denote the parts of $G$, a model of $T$, by $A$ and $B$, by abuse of notation.

We say that a graph $G$ has \emph{bipartite $k$-extension property} if it satisfies the following:
\begin{enumerate}
    \item for any disjoint $X, Y \subseteq A$, $Z \subseteq B$ with $|X|+|Y|+|Z| \leq k-1$ there exists $v \in B \setminus Z$ such that $\forall x \in X: vx \in E$ and $\forall y \in : vy \not\in E$,
    \item for any disjoint $X, Y \subseteq B$, $Z \subseteq A$ with $|X|+|Y|+|Z| \leq k-1$ there exists $v \in A \setminus Z$ such that $\forall x \in X: vx \in E$ and $\forall y \in : vy \not\in E$.
\end{enumerate}

A countable model of $T$ that has bipartite $k$-extension property for every $k \in \N$ is called \emph{bipartite Rado graph} whose properties are reminiscent of the Rado graph \cite{hodges} (also known as the countable random graph).
A standard back-and-forth argument shows that there is only one bipartite Rado graph $\mathcal{BR}$ up to isomorphism and it is ultrahomogeneous.
Moreover, if $(G_n)$ is a sequence of models of $T$ with increasing size of parts and for each $k \in \N$ there is an index $n_k$ such that for $n \geq n_k$ the graph $G_n$ has the bipartite $k$-extension property, then $(G_n)$ elementarily converges to $\mathcal{BR}$.
We recall that if a sequence $(G_n)$ elementarily converges to a an ultrahomogeneous structure, the question of $\FO$-convergence reduces to $\QF$-convergence \cite[Lemma 2.28]{unified_approach}.

Let $G_n(p) = (A_n \cup B_n, E_n)$, $p \in (0,1)$, be a model of $T$ with parts of size $n^2$ and $n$ with the edge between each pair $u \in A_n, v \in B_n$ with probability $p$, independently of all the other pairs.
A direct computation yields that for each $k \in \N$ the probability that $G_n(p)$ does not posses the bipartite $k$-extension property decays exponentially with $n$.
Therefore, using Borel-Cantelli lemma, the sequence $(G_n(p))$ elementarily converges to $\mathcal{BR}$ almost surely.

Observe that $(G_n(p))$ is \emph{always} $\QF$-convergent as almost all $t$-tuples of vertices induce an independent set.
This also implies that the sequence $(H_n)$ formed by interlacing $(G_n(p))$ and $(G_n(q))$ for some $0 < p < q < 1$ is almost surely $\FO$-convergent.

We claim that for any sequence of roots $(r_n)$ from the smaller components $B_n$ the sequence $(H_n, r_n)$ does not converge.
Fix $\eps > 0$ and consider the event ``$\forall u \in B : \deg u \in pn^2 \pm \eps$''.
This event holds for all except finitely many graphs of $(G_n(p))$ by Chernoff bounds and Borel-Cantelli lemma.
Therefore, the proportion of neighbors of roots $r_n$ from $H_n$ oscillates between $pn^2 \pm \eps$ and $qn^2 \pm \eps$.
Therefore, using $\eps = (q-p)/3$, the sequence $(H_n, r_n)$ is almost surely not $\FO$-convergent as witnessed by the formula $\phi(x): x \sim \Root$.

If we want to remove the marks, we can attach e.g. a triangle to each vertex of the smaller part and a pentagon to each vertex of the larger part (i.e. preserving the property that the parts are definable).
This operation can be formalized as a gadget construction, the marks are considered to be the replaced edges by the gadgets triangle and pentagon, which in these cases preserves $\FO$-convergence \cite[Theorem~5.3, Corollary~5.5]{gadget_construction}.
\section{Concluding remarks}\label{sec:concluding_remarks}

An iterative use of Theorem~\ref{thm:rooting_algebraic_vertices}~or~\ref{thm:rooting_solution_of_algebraic_formulas} allows us to gain complete control over the algebraic elements as we can consider each of them separately.
Note that it is also possible to root solutions of algebraic formulas with multiple free variables (i.e. $p$-tuples instead of singletons) since the projection to each coordinate yields an algebraic set.

We would like to point out that Theorem~\ref{thm:rooting_solution_of_algebraic_formulas} remains valid for $\FO$-convergent sequences $(G_n)$ without a modeling limit.
The proofs are analogous except that the set $I$ in Lemma~\ref{lem:rooting_for_finite_collection_of_formulas} is defined as an inclusion-maximal set for which there are roots $(r_n)$ with the property that $\lim \stonepar{\bigland_{i \in I} \phi_i(\tpl{x}_i)}{(G_n, r_n)} > 0$. 

It can be shown that the sequence of random bipartite graphs constructed for Proposition~\ref{prop:example} admits (almost surely) a modeling limit.
This together with a rich context will be the subject of a forthcoming paper.

Besides the original problem in \cite{unified_approach}, our motivation was the study of structural convergence of sequences created via gadget construction, see \cite{gadget_construction}.
Using the result of this paper, we conclude that $\FO$-convergence is preserved by gadget construction if the gadgets replace only finitely many edges (under additional natural assumptions).

In the typical case, the modeling $L$ is of size continuum and the set of algebraic vertices (which is at most countable) has measure $0$. Hence, our results reveal only a negligible portion of vertices of $L$ for which the roots $(r_n)$ exist, which shows that there is still room for further research.

\printbibliography


\end{document}